\documentclass[12pt]{article}
\usepackage{e-jc_for_arxiv}

\usepackage[T1]{fontenc}
\usepackage[utf8]{inputenc}
\usepackage{amsmath}
\usepackage{amsthm}
\usepackage{amssymb}
\usepackage{mathrsfs}
\usepackage{graphicx}
\usepackage[shortlabels]{enumitem}
\usepackage{mathtools}
\usepackage[usenames,dvipsnames,svgnames,table]{xcolor}
\usepackage{tikz}
\tikzset{black node/.style={draw, circle, fill = black, minimum size = 5pt, inner sep = 0pt}}
\tikzset{normal/.style = {draw=none, fill = none, minimum size =0, rectangle}}
\usepackage[longnamesfirst,numbers,sort&compress]{natbib}
\makeatletter
\def\NAT@spacechar{~}
\makeatother
\usepackage{hyperref}
\usepackage[hyphenbreaks]{breakurl}
\hypersetup{
breaklinks=true, 
colorlinks=true,
linkcolor=black,
urlcolor={blue!60!black},
citecolor=black,
filecolor=black,
pdftitle={Smaller Extended Formulations for Spanning Tree Polytopes in Minor-closed Classes and Beyond},
pdfauthor={Manuel Aprile, Samuel Fiorini, Tony Huynh, Gwena\"{e}l Joret, David R. Wood}}

\renewcommand{\geq}{\geqslant}
\renewcommand{\leq}{\leqslant}

\newcommand{\defn}[1]{\textcolor{Maroon}{\emph{#1}}}
\DeclareMathOperator{\conv}{conv}
\DeclareMathOperator{\xc}{xc}
\DeclareMathOperator{\STP}{P}
\DeclareMathOperator{\FP}{P^\downarrow}

\newcommand{\R}{\mathbb R}

\newcommand{\N}{\mathbb N}

\newcommand{\GG}{\mathcal{G}}

\dateline{Jun 22, 2021}{Nov 30, 2021}{TBD}

\title{Smaller Extended Formulations for Spanning Tree Polytopes in Minor-closed Classes and Beyond}

\author{
Manuel Aprile\footnotemark[2]
\qquad  Samuel Fiorini\footnotemark[3]
\qquad Tony Huynh\footnotemark[4]\\
 Gwena\"{e}l Joret\footnotemark[5] 
\qquad David R. Wood\footnotemark[4] 
}

\footnotetext[2]{Mathematics Department, Università degli studi di Padova, Padova, Italy (\texttt{manuelf.aprile@gmail.com})}

\footnotetext[3]{Mathematics Department, Universit\'e libre de Bruxelles, Brussels, Belgium (\texttt{sfiorini@ulb.ac.be})}

\footnotetext[4]{School of Mathematics, Monash University, Melbourne, Australia (\texttt{tony.bourbaki@gmail.com,david.wood@monash.edu})}

\footnotetext[5]{Computer Science Department, Universit\'e libre de Bruxelles, Brussels, Belgium (\texttt{gjoret@ulb.ac.be})}

\begin{document}
\maketitle

\begin{abstract}
Let $G$ be a connected $n$-vertex graph in a proper minor-closed class $\mathcal G$.  We prove that the extension complexity of the spanning tree polytope of $G$ is $O(n^{3/2})$.  This improves on the $O(n^2)$ bounds following from the work of Wong~(1980) and Martin~(1991).  It also extends a result of Fiorini, Huynh, Joret, and Pashkovich~(2017), who obtained a $O(n^{3/2})$ bound for graphs embedded in a fixed surface. Our proof works more generally for all graph classes admitting strongly sublinear balanced separators: We prove that for every constant $\beta$ with $0<\beta<1$, if $\mathcal G$ is a graph class  closed under induced subgraphs such that all $n$-vertex graphs in $\mathcal G$ have balanced separators of size $O(n^\beta)$, then the extension complexity of the spanning tree polytope of every connected $n$-vertex graph in $\mathcal{G}$ is $O(n^{1+\beta})$. 
We in fact give two proofs of this result, one is a direct construction of the extended formulation, the other is via communication protocols. Using the latter approach we also give a short proof of the $O(n)$ bound for planar graphs due to Williams~(2002).
\end{abstract}

\newpage
\section{Introduction}
An \defn{extended formulation} of a (convex) polytope $P \subseteq \R^d$ is a linear system $Ax + By \leqslant b,\ Cx + Dy = c$ in variables $x \in \R^d$ and $y \in \R^k$ such that
$$
P = \{x \in \R^d \mid \exists y \in \R^k : Ax + By \leqslant b,\ Cx + Dy = c\}\,.
$$
The \defn{size} of an extended formulation is its number of inequalities. The \defn{extension complexity} $\xc(P)$ is the minimum size of an extended formulation of $P$. Note that equalities are not counted in the size of an extended formulation. Equivalently, the extension complexity of $P$ is the minimum number of facets of a polytope that affinely projects to $P$.
One motivation for studying extension complexity is that if $P$ has many facets but small extension complexity, then it can be much faster to optimize over an extended formulation of $P$ rather than over $P$ itself.  

A classic example where moving to a higher dimensional space dramatically reduces the number of facets is the spanning tree polytope of a graph, which we now define. Let $G$ be a connected (simple, finite, undirected) graph.  The \defn{spanning tree polytope} of $G$, denoted $\STP(G)$, is the convex hull of the $0/1$ characteristic vectors in $\R^{E(G)}$ of the spanning trees of $G$.

Edmonds'~\cite{Edmonds71} description of the spanning tree polytope in $\R^{E(G)}$ has exponentially many facets. However, \citet{Wong80} and \citet{Martin91} proved that $\xc(\STP(G)) \in O(|V(G)| \cdot |E(G)|)$ for every connected graph $G$.  Since the extension complexity of a polytope is at least its dimension, $\xc(\STP(G)) \geq |E(G)|$, which is the only known lower bound.  A notoriously difficult problem of Michel Goemans asks to improve either of these bounds, but this is still wide open; see \citet{KT18}.

Improved upper bounds are known for restricted graph classes. For example,  \citet{Williams2002}\footnote{See \cite{FHJP17,VB19} for a correction to Williams' proof.} proved that $\xc(\STP(G)) \in O(n)$ for every connected $n$-vertex planar graph $G$. A $O(n)$ bound also holds for graphs of bounded treewidth~\cite{KKT20} and for graphs that can be made planar by deleting a bounded number of vertices~\cite{FHJP17}. \citet{FHJP17} proved that for connected $n$-vertex graphs $G$ with bounded Euler genus, $\xc(\STP(G)) \in O(n^{3/2})$. 

In this paper, we generalise this result to all proper minor-closed graph classes\footnote{A graph $H$ is a \defn{minor} of a graph $G$ if $H$ is isomorphic to a graph obtained from a subgraph of $G$ by contracting edges. A \defn{graph class} is a family of graphs that is closed under isomorphism. A graph class $\GG$ is \defn{minor-closed} if $H\in \GG$ whenever $G\in\GG$ and $H$ is a minor of $G$. We say that $\GG$ is \defn{closed under induced subgraphs} if $H\in \GG$ whenever $G\in\GG$ and $H$ is an induced subgraph of $G$.  Finally, $\GG$ is \defn{proper} if some graph is not in $\GG$.}.  Instead of the spanning tree polytope, we primarily work with the \defn{forest polytope} of $G$, which is the convex hull of the $0/1$ characteristic vectors in $\R^{E(G)}$ of the forests of $G$. We denote this polytope $\FP(G)$. Since $\STP(G)$ is equal to $\FP(G)$ intersected with the hyperplane $\{x \in \mathbb{R}^{E(G)} \mid \sum_{e \in E(G)}x_e=|V(G)|-1\}$, we have $\xc(\STP(G)) \leq \xc(\FP(G))$. Therefore, every upper bound on $\xc(\FP(G))$ yields the same upper bound on $\xc(\STP(G))$.  

\begin{theorem} \label{minorclosed}
For every proper minor-closed graph class $\GG$, the forest polytope of every $n$-vertex graph in $\GG$ has extension complexity in $O(n^{3/2})$.
\end{theorem}

In fact, we prove a stronger theorem for all graph classes with strongly sublinear separators (precise definitions will be given later).  

\begin{theorem} \label{main}
Let $\GG$ be a graph class closed under induced subgraphs, and $\beta \in (0,1)$ be such that every $n$-vertex graph in $\GG$ has a $\frac{1}{2}$-balanced separator of size $O(n^{\beta})$.  Then the forest polytope of every   $n$-vertex graph in $\GG$ has extension complexity in $O(n^{1+\beta})$.
\end{theorem}

By a result of \citet{AST90}, for proper minor-closed classes, we may take $\beta=\frac{1}{2}$ in Theorem~\ref{main}. Therefore, Theorem~\ref{main} immediately implies Theorem~\ref{minorclosed}. Moreover, Theorem~\ref{main} is applicable for many graph classes that are not minor-closed and of independent interest. Indeed, by a result of \citet{DN16}, Theorem~\ref{main} is applicable for any graph class with polynomial expansion.  For all such graph classes, Theorem~\ref{main} gives better bounds on $\xc(\STP(G))$ than the $O(n^2)$ bound following from~\cite{Wong80, Martin91}.  These connections and examples are presented in Section~\ref{sec:sublinear}.

We give two proofs of Theorem~\ref{main}. The first proof directly constructs the extended formulation and is very short; see Section~\ref{DirectProof}. The second proof exploits an equivalent definition of extension complexity using randomized protocols from communication complexity~\cite{FFGT15}.  We describe this equivalence in Section~\ref{sec:protocols}, where we also present the classical protocol for the forest polytope of general graphs~\cite{FFGT15}. In Section~\ref{sec:proof}, we present the alternative proof of our main theorem via a randomized protocol. Note that we in fact found the latter proof first, and then derived the  direct proof from it. We believe that the proof via randomized protocols offers an alternative perspective of independent interest. Furthermore, the framework of randomized protocols might turn out to be helpful when attacking some of the remaining open problems in the area.  In this spirit, we present an alternative and short proof via randomized protocols of the result by  \citet{Williams2002} showing that $\xc(\STP(G)) \in O(n)$ for every connected $n$-vertex planar graph $G$; see Section~\ref{sec:planar}. Finally, we conclude in Section~\ref{sec:open_problems} with some open problems.

\section{Sublinear Separators and Polynomial Expansion} 
\label{sec:sublinear}

Let $\alpha \in (0, 1)$.  An \defn{$\alpha$-balanced separator} in a graph $G$ is a set $X\subseteq V(G)$ such that $G-X$ is the disjoint union of two induced subgraphs $G_1$ and $G_2$ with $|V(G_1)|, |V(G_2)| \leq \alpha|V(G)|$.  A graph class $\GG$ has \defn{strongly sublinear separators} if there exist $\beta\in(0,1)$ and $c \in \mathbb{R}$ such that every graph $G\in \GG$ has a $\frac{1}{2}$-balanced separator of size at most $c|V(G)|^{\beta}$. This definition is typically stated with  $\frac{2}{3}$ instead of $\frac{1}{2}$, but the following folklore lemma shows that our definition is equivalent (up to a constant factor on the size of the separator). For completeness we include a proof. 

\begin{lemma}
\label{sublinear}
Let $\GG$ be a graph class closed under induced subgraphs such that for some constants $\alpha,\beta\in(0,1)$ and $c>0$, every graph $G\in\GG$ has an $\alpha$-balanced separator of size at most $c\,|V(G)|^\beta$. Then every graph $G\in\GG$ has a $\frac{1}{2}$-balanced separator of size at most $\frac{c}{1-\alpha}\,|V(G)|^\beta$. 
\end{lemma}

\begin{proof}
We first prove by induction on $|V(G)|$ that each $G\in \GG$ has pathwidth at most 
$\frac{c}{1-\alpha} |V(G)|^\beta-1$. That is, there is a sequence $B_1,\dots,B_n$ of non-empty subsets of $V(G)$ such that:
\begin{enumerate}[(1),itemsep=0ex,topsep=0ex]
\item $B_1\cup \dots\cup B_n=V(G)$, 
\item if $1\leq i<j<k\leq n$ then $ B_i\cap B_k \subseteq B_j$,  
\item for each edge $vw$ of $G$, there exists $i\in[n]$ such that $v,w\in B_i$, and
\item $|B_i| \leq \frac{c}{1-\alpha} |V(G)|^\beta$ for each $i\in[n]$. 
\end{enumerate}
If $|V(G)|=1$ then $B_1:=V(G)$ is the desired sequence. Now assume that $|V(G)|\geq 2$. By assumption, there is a set $X\subseteq V(G)$ of size at most $c|V(G)|^\beta$ such that $G-X$ is the disjoint union of two induced subgraphs $G_1$ and $G_2$ with $|V(G_1)|, |V(G_2)| \leq \alpha |V(G)|$. By induction, $G_1$ has the desired sequence $A_1,\dots,A_a$ and $G_2$ has the desired sequence $B_1,\dots,B_b$. The sequence $A_1\cup X,\dots,A_a\cup X,B_1\cup X,\dots,B_b\cup X$ satisfies (1)--(3) by construction, and it satisfies (4) since
$$|A_j\cup X|,|B_j\cup X| 
\leq \tfrac{c}{1-\alpha} |V(G_i)|^\beta + c|V(G)|^\beta \leq 
(\tfrac{c\alpha}{1-\alpha} + c)|V(G)|^\beta =
\tfrac{c}{1-\alpha} |V(G)|^\beta.$$
Let $B_1,\dots,B_n$ be a sequence satisfying (1)--(4) for any $G\in\GG$. Suppose that there are distinct vertices $v,w\in B_{i} \setminus B_{i+1}$ for some $i\in[n-1]$. 
Then replace $B_i,B_{i+1}$ in the sequence by $B_i,B_i\setminus\{v\},B_{i+1}$. 
The new sequence satisfies (1)--(4). Repeat this operation (and the symmetric operation)
until we obtain a sequence $B_1,\dots,B_m$ that satisfies (1)--(4) and:
\begin{enumerate}[(5),itemsep=0ex,topsep=0ex]
\item $|B_{i} \setminus B_{i+1}|\leq 1$ and $|B_{i+1} \setminus B_{i}|\leq 1$ for each $i\in[m-1]$. 
\end{enumerate}
For each $i\in[m]$, let $X_i:=\bigcup\{B_j\setminus B_i:j<i\}$ and $Y_i:= \bigcup\{B_j\setminus B_i:j>i\}|\}$. So $X_i,B_i,Y_i$ is a partition of $V(G)$. By (2) and (3), no edge of $G$ has one endpoint in $X_i$ and the other endpoint in $Y_i$. Thus $G-B_i$ is the disjoint union of the induced subgraphs $G[X_i]$ and $G[Y_i]$.
For each $i\in[m-1]$, we have $X_{i}\subseteq X_{i+1}$ by property (2), and 
$|X_{i+1}|\leq |X_{i}|+1$ by property (5). So $|X_1|,\dots,|X_m|$ is a non-decreasing sequence from $0$ to $|X_m|$ increasing by at most 1 at each step. By symmetry, $|Y_1|,\dots,|Y_m|$ is a non-increasing sequence from $|Y_1|$ to $0$ decreasing by at most 1 at each step.
Consider the piecewise linear continuous function defined by  $f(i):=|X_i|-|Y_i|$. By the Intermediate Value Theorem, $||X_i|-|Y_i||\leq 1$ for some $i\in[m]$. Thus $G-B_i$ is the disjoint union of $G_1:=G[X_i]$ and $G_2:=G[Y_i]$, where $|V(G_1)|,|V(G_2)|\leq\frac12 |V(G)|$ (since $B_i\neq\emptyset$).
\end{proof}

Graph classes with strongly sublinear separators are characterised  via the notion of polynomial expansion, due to \citet{Sparsity}.  The \defn{density} of a non-empty graph $G$ is $\frac{|E(G)|}{|V(G)|}$.  A graph class $\GG$ has \defn{bounded expansion} with \defn{expansion function} $f:\mathbb{Z}^+\to\mathbb{R}$ if, for every graph $G\in\GG$ and for all pairwise-disjoint subgraphs $B_1,\dots,B_t$ of radius at most $r$ in $G$, the graph obtained from $G$ by contracting each $B_i$ into a vertex has density  at most $f(r)$. When $f(r)$ is a constant, $\GG$ is contained in a proper minor-closed class. As $f(r)$ is allowed to grow with $r$ we obtain larger and larger graph classes.  A graph class $\GG$ has \defn{polynomial expansion} if $\GG$ has bounded expansion with an expansion function in $O(r^c)$, for some constant $c$. 
\citet{DN16} characterised graph classes with polynomial expansion as follows. 

\begin{theorem}[\cite{DN16}] \label{expansion}
A graph class $\GG$ closed under induced subgraphs has strongly sublinear separators if and only if $\GG$ has polynomial expansion.  
\end{theorem}

Theorems~\ref{main} and \ref{expansion} imply:

\begin{corollary}
For every graph class $\GG$ closed under induced subgraphs and with polynomial expansion, there exists $\beta\in(0,1)$ such that the forest polytope of every $n$-vertex   graph in $\GG$ has extension complexity in $O(n^{1+\beta})$.
\end{corollary}

We in fact prove the following more precise version of Theorem~\ref{main}.

\begin{theorem} \label{thm:main2}
Let $\GG$ be a graph class closed under induced subgraphs such that for some  $c,d \in \mathbb{R}^+$ and $\beta\in(0,1)$,  every $n$-vertex graph in $\GG$ has density at most $d$ and has a $\frac{1}{2}$-balanced separator of size at most $cn^{\beta}$. Then the forest polytope of every $n$-vertex    graph in $\GG$ has extension complexity in $O(cdn^{1+\beta})$.
\end{theorem}

Theorem~\ref{thm:main2} implies Theorem~\ref{main} since graph classes admitting strongly sublinear separators have polynomial expansion by Theorem~\ref{expansion}, and hence have bounded density. 

We now present several examples of Theorem~\ref{thm:main2}.

\subsection*{Minor-closed classes} 
Let $G$ be an $n$-vertex   $K_t$-minor-free graph. Kostochka~\cite{kostochka1984lower} and Thomason~\cite{Thomason1984} independently proved that $G$ has density $O(t (\log t)^{1/2})$.  Kawarabayashi and Reed~\cite{KR10} proved that $G$ has a $\frac{2}{3}$-balanced separator of size $O(tn^{1/2})$ (improving on the original $O(t^{3/2}n^{1/2})$ bound of Alon, Seymour and Thomas~\cite{AST90}). By Lemma~\ref{sublinear}, for $K_t$-minor-free graphs, we may take $c=O(t)$, $d=O(t (\log t)^{1/2})$, and $\beta=\frac{1}{2}$ in Theorem~\ref{thm:main2}, which gives the following more precise version of Corollary~\ref{minorclosed}.

\begin{corollary}
\label{minorclosed2}
The forest polytope of every $n$-vertex   $K_t$-minor-free graph has extension complexity in $O(t^2 (\log t)^{1/2} n^{3/2})$.  
\end{corollary} 

\subsection*{Bounded genus}
Let $G$ be an $n$-vertex   graph of Euler genus $g$. Such graphs have balanced separators of size $O(\sqrt{gn})$ (see \cite{AS96,DMW17,GHT84}) and density less than $3+\frac{3g}{n}$ by Euler's formula. Thus Theorem~\ref{thm:main2} implies the forest polytope of $G$ has extension complexity in 
$O(g^{1/2}n^{3/2}+g^{3/2}n^{1/2})$, which matches the bound proved by Fiorini, Huynh, Joret and Pashkovich~\cite{FHJP17}. 

\subsection*{Bounded crossings}
A graph is \defn{$(g,k)$-planar} if it has a drawing in a surface with Euler genus $g$ with at most $k$ crossings on each edge. Note that the class of $(g, k)$-planar graphs is \emph{not} minor-closed, even in the $g=0$ and $k=1$ case~\cite{DEW17}.  However, every $n$-vertex $(g,k)$-planar graph has density  $O(gk)$ and has a balanced separator of size $O(\sqrt{gkn})$; see \cite{DEW17,DMW17}. By Theorem~\ref{thm:main2}, the forest polytope of such graphs has extension complexity in $O((gkn)^{3/2})$.


\subsection*{Intersection graphs of balls}
Let \defn{$\GG_{d,k}$} be the class of intersection graphs of a set of balls in $\mathbb{R}^d$, where each point in $\mathbb{R}^d$ is in at most $k$ of the balls. \citet{MTSTV97} showed that each $n$-vertex graph $G\in \GG_{k,d}$ has a $(1-\frac{1}{d+2})$-balanced separator of size $O(k^{1/d}\,n^{1-1/d})$ and has density at most $3^dk$. By Lemma~\ref{sublinear}, each $n$-vertex graph $G\in \GG_{k,d}$ has a $\frac{1}{2}$-balanced separator of size $O(k^{1/d}(d+2)\,n^{1-1/d})$. By Theorem~\ref{thm:main2}, the forest polytope of each $n$-vertex graph $G \in \GG_{d,k}$ has extension complexity in  $O(k^{1+1/d}(d+2)3^d\,n^{1-1/d})$. Numerous other  intersection graphs of certain geometric objects admit strongly sublinear separators~\cite{DMN,SW98,Lee16}; Theorem~\ref{thm:main2} is applicable in all these settings.

\section{Direct Proof}
\label{DirectProof}

In this section we give our first proof of Theorem~\ref{thm:main2}. We need Edmond's linear description of the forest polytope~\cite{Edmonds71}.

\begin{theorem}[\cite{Edmonds71}] \label{thm:LP2}
For every graph $G$,
\[
  \FP(G) = \begin{array}[t]{r@{\ }l}
    \big \{x\in \mathbb{R}^{E(G)}_{\geq 0} : 
    &x(E(U))\leq |U|-1, \; \forall\, U\subsetneq V(G),\, U\neq\emptyset \big \}\,.
  \end{array}
\]
\end{theorem}

In the above description, $x(F)$ denotes $\sum_{e \in F} x_e$ and $E(U)$ denotes the set of edges of $G$ with both ends in $U$.  

We start with a known ``decomposition'' result: whenever $G$ is disconnected, the forest polytope of $G$ is the Cartesian product of the forest polytopes of its components. The same holds for the spanning tree polytope of $G$. Actually, this holds more generally when we consider the blocks\footnote{Recall that a \defn{block} of a graph $G$ is an induced subgraph $H$ of $G$ such that $H$ is either $2$-connected or isomorphic to $K_1$ or $K_2$, and $H$ is inclusion-wise maximal with this property.} of $G$. We state the result in this more general form, which is used in Section~\ref{sec:planar}.

\begin{lemma}\label{lem:2conn}
For every graph $G$ with blocks $G_1, \dots, G_k$, 
$$\FP(G) = \FP(G_1) \times \cdots \times \FP(G_k)
\text{ and }
\STP(G)=\STP(G_1)\times \dots \times \STP(G_k).$$
\end{lemma}

\begin{proof}
If $G$ is 2-connected, then there is nothing to prove. Otherwise, the lemma  follows by iteratively applying the following observation. Let $G_1$ and $G_2$ be induced subgraphs of $G$ with at most one vertex in common, and whose union is $G$. Then a subgraph $T$ of $G$ is a forest (respectively, spanning tree) of $G$ if and only if each $G_i$ has a forest (resp. spanning tree) $T_i$ with $T=T_1\cup T_2$. 
\end{proof}

Lemma~\ref{lem:vtx_deletion} below quantifies the change in extension complexity for the forest polytope when a vertex is deleted from the graph. The variables $z_{(v,w)}$ used in the definition of $Q(G)$ are identical to those of Martin's extended formulation for the spanning tree polytope~\cite{Martin91}, except that here it suffices to consider a single root. 

\begin{lemma} \label{lem:vtx_deletion}
Let $G$ be a graph, and let $r$ be an arbitrary vertex of $G$. Let $G^+$ denote the graph obtained from $G$ by adding one new vertex adjacent to every vertex of $G$, and let $A(G^+) := \{(v,w) : vw \in E(G^+)\}$. 
Let
$$
Q(G) := \big\{x \in \R^{E(G)} \mid \exists z \in \R^{A(G^+)}_{\geq 0} : \begin{array}[t]{l}
z_{(r,v)} = 0, \, \forall v \in N(r),\\
x_{vw} = z_{(v,w)} + z_{(w,v)}, \, \forall vw \in E(G),\\ \sum_{w \in N(v)} z_{(v,w)} = 1, \, \forall v \in V(G^+ - r)\big\},
\end{array}
$$
where the neighborhoods are computed in $G^+$. Then 
\begin{equation}
\label{eq:vtx_deletion}
\FP(G) = (\FP(G-r) \times \R^{\delta(r)}) \cap Q(G)\,.
\end{equation}
In particular, 
$$
\xc(\FP(G)) \leq \xc(\FP(G-r)) + 2 (|E(G)| + |V(G)|)\,.
$$
\end{lemma}

\begin{proof}
Let $R(G)$ denote the right hand-side of \eqref{eq:vtx_deletion}. We prove that $\FP(G) \subseteq R(G)$ directly from the definitions, and then that $R(G) \subseteq \FP(G)$ using Theorem~\ref{thm:LP2}.

Let $F \subseteq E(G)$ be  any forest in $G$, and let $x := \chi^F$. 
First, notice that since $F \cap E(G-r)$ is a forest of $G-r$, the restriction of $x$ to $\R^{E(G-r)}$ is in $\FP(G-r)$.
Second, let $T$ be any spanning tree of $G^+$ such that $T \cap E(G) = F$. We root $T$ at $r$. For $(v,w) \in A(G^+)$, we let $z_{(v,w)} := 1$ if the parent of $v$ (in $T$ rooted at $r$) is $w$, and $z_{(v,w)} := 0$ otherwise. The resulting point $z \in \R^{A(G^+)}$ satisfies all the constraints in the extended formulation defining $Q(G)$. This proves that $x \in Q(G)$. Hence, $\FP(G)$ is contained in $R(G)$.

Next, observe that every point $x \in R(G)$ satisfies $x_{vw} \geq 0$ for every edge $vw \in E(G)$. Hence, $R(G) \subseteq \R^{E(G)}_{\geq 0}$. It suffices to prove that the rank inequality $x(E(U)) \leqslant |U| - 1$ is valid for $R(G)$, for all non-empty proper subsets $U$.

If $U$ does not contain $r$, then $x(E(U)) \leqslant |U| - 1$ is valid for $R(G)$ since it is valid for $\FP(G-r)$.

Now, assume that $r \in U$. We claim that $x(E(U)) \leqslant |U| - 1$ is valid for $Q(G)$. This of course implies that it is valid for $R(G)$. Let $x \in Q(G)$ and let $z \in \R^{A(G^+)}$ be a point that witnesses this. Then,
\begin{align*}
x(E(U)) &= \sum_{vw \in E(U)} \underbrace{(z_{(v,w)} + z_{(w,v)})}_{=x_{vw}}\\
&\leqslant \sum_{vw \in E(U)} (z_{(v,w)} + z_{(w,v)})
+ \sum_{(v,w) \in A(G^+), v \in U \setminus \{r\}, w \notin U} \underbrace{z_{(v,w)}}_{\geq 0}\\
& = \sum_{v \in U \setminus \{r\}} \underbrace{\sum_{w \in N(v)} z_{(v,w)}}_{=1} = |U| - 1\,,
\end{align*}
implying the claim. This concludes the proof.
\end{proof}

The next lemma follows directly from Lemmas~\ref{lem:2conn} and \ref{lem:vtx_deletion}.

\begin{lemma} \label{lem:separator}
Let $G$ be a graph, let $X \subseteq V(G)$, and let $G_1$ and $G_2$ be vertex-disjoint induced subgraphs of $G$ whose union is $G - X$. Then 
$$
\xc(\FP(G)) \leq \xc(\FP(G_1)) + \xc(\FP(G_2)) + 2 |X| \cdot (|E(G)| + |V(G)|)\,.
$$
\end{lemma}

We are now ready to prove our main theorem.

\begin{proof}[Proof of Theorem~\ref{thm:main2}]
For the sake of simplicity, we assume that $d \geq 1$. (Otherwise all graphs in graph class $\GG$ with at least $1/(1-d)$ vertices are disconnected.) For a positive integer $n$, let $f(n)$ denote the maximum extension complexity of the forest polytope of an $n$-vertex graph in $\GG$. For small values of $n$, we may resort to the trivial bound $f(n) \leqslant 2^{dn}$, which follows directly from Theorem~\ref{thm:LP2}. Letting $X$ be a $\frac12$-balanced separator $X$, Lemma~\ref{lem:separator} implies
$$
f(n) \leq 2 f(\lfloor \tfrac{n}{2} \rfloor) + 2(dn + n) c n^{\beta}
\leq 2 f(\lfloor \tfrac{n}{2} \rfloor) + 3 cd n^{1+\beta}\,.
$$
Since $\beta > 0$, the Master Theorem gives $f(n) \in O(cd n^{1+\beta})$. Indeed,  $f(n)=\frac{3}{1-2^{-\beta}} cdn^{1+\beta}$ satisfies this recurrence.
\end{proof}

Assuming that a $\frac12$-balanced separator $X \subseteq V(G)$ can be found in polynomial time for every $n$-vertex graph $G$ in $\GG$, the extended formulation for $\FP(G)$ given by Theorem~\ref{thm:main2} can also be written down in polynomial time. It is enough to find a $\frac23$-balanced separator in polynomial time, since the proof of Lemma~\ref{sublinear} is algorithmic. For $n$-vertex graphs as in Theorem~\ref{thm:main2}, a $\frac23$-balanced separator of size $O(cdn^{1+\beta} \sqrt{\log n})$ can be found in polynomial time using an algorithm of~\citet{FHL08}, which is thus only slightly larger than the best possible size of $O(cdn^{1+\beta})$. 
For some classes $\mathcal G$, it is known that the extra $O(\sqrt{\log n})$ factor in the separator size can be avoided, most notably for proper-minor closed classes~\cite{AST90,KR10,RW09}. 

\section{Communication Protocols} 

\subsection{Randomized Protocols} 
\label{sec:protocols}

This section describes an equivalent definition of extension complexity via the communication complexity of a certain two player game~\cite{FFGT15,zhang2012quantum}.  For $n\in\N$, let $[n]:=\{1,2,\dots,n\}$. Let $P = \conv(\{v_1,\dots, v_n\})=\{x\in \mathbb{R}^d \mid Ax \leq b\}$, where $A \in \R^{m \times d}$, $b \in \R^m$. The \defn{slack matrix} associated with these two descriptions 
of $P$ is the matrix $S \in \R^{m \times n}_+$ where $S_{i,j} := b_i-A_i v_j$ for $i\in [m]$ and $j\in [n]$.  That is, $S_{i,j}$ is the slack of point $v_j$ with respect to the inequality $A_i x \leq b_i$. The concept of slack matrix was introduced by Yannakakis~\cite{yannakakis1991expressing} and is inherently related to the extension complexity of a polytope. In particular, extended formulations for a polytope can be obtained from a deterministic protocol computing its slack matrix~\cite{yannakakis1991expressing}. This was extended by \citet{FFGT15} to randomized protocols, as we now describe.

Let $S$ be a non-negative matrix (in our setting, $S$ is a slack matrix of a given polytope). Consider two agents Alice and Bob. Alice is given as input a row index $i$ of $S$, and Bob is given a column index $j$. A \defn{randomized protocol} is a process during which Alice and Bob, given their inputs, exchange information, and at the end output a non-negative number. At each step the information sent by each player may depend on their input, on the information exchanged so far, and on an unlimited amount of random bits that each player can use. Hence, the output $s_{ij}$ of the protocol on inputs $i$, $j$ is a random variable. The protocol is said to \defn{compute $S$ in expectation} if the expectation of $s_{ij}$ is equal to $S_{i,j}$ for each row index $i$ and column index $j$ of $S$. The \defn{complexity} of a randomized protocol is the maximum number of bits exchanged by Alice and Bob in an execution of the protocol. The following result of \citet{FFGT15} establishes the equivalence between extended formulations and randomized protocols (all logarithms in this paper are binary).

\begin{theorem}[\cite{FFGT15}]
\label{thm:random}
For every polytope $P$ with at least two vertices, the minimum complexity of a randomized protocol computing a slack matrix of $P$ in expectation equals $\lceil \log (\xc(P)) \rceil$.
\end{theorem}

To give the reader some intuition about randomized protocols we now briefly describe the randomized protocol from~\cite{FFGT15}, which gives an alternative proof of the Martin--Wong bound on the extension complexity of the forest polytope.  We refer to this randomized protocol as the \defn{classical protocol}.

Let $S_G$ be the matrix with columns indexed by the forests $F$ of $G$, and rows indexed by proper, non-empty subsets $U$ of $V(G)$, where 
\[
S_G(U,F)=|U|-1-|E(F)\cap E(U)|.
\]
Note that $S_G$ is a submatrix of the slack matrix of $\FP(G)$ with respect to the linear description given in Theorem~\ref{thm:LP2}. As  follows easily from basic facts on extended formulations, any protocol computing $S_G$ in expectation with complexity $c$ can be turned into an extended formulation for $\FP(G)$ of size at most $2^c + |E(G)|$. Indeed, there are only $|E(G)|$ additional inequalities besides the rank inequalities $x(E(U)) \leq |U| - 1$, which are captured in $S_G$. Because $\xc(\FP(G)) \geq |E(G)|$, it suffices to give a randomized protocol that computes $S_G$ in expectation. 

We now describe the classical protocol. Alice receives as input a non-empty set $U \subsetneq V(G)$, and Bob receives a forest $F$ of $G$. Their goal is to compute $S_G(U,F)$ in expectation.  
It will be helpful to focus on spanning trees instead of forests. 
To do so, we add a new vertex $x$ to $G$ and make it adjacent to all vertices of $G$, and denote the resulting graph by $G^+$. 
Then, Bob extends his forest $F$ into a spanning tree $T$ of $G^+$ by adding some of the edges incident to $x$. 
Since every edge added is incident to $x$, none of the added edges have both ends in $U$. 
Therefore, $S_{G^+}(U,T)=S_{G}(U, F)$. 
Observe also that $S_{G^+}(U,T)=k-1$ where $k$ is the number of connected components of $G^+[U] \cap T$.

 Alice begins by sending any vertex $u \in U$ to Bob.  Bob roots $T$ at $u$ and orients the edges of $T$ away from $u$.  Bob then samples an arc $a$ of $T$ uniformly at random and sends $a$ to Alice.  Alice outputs $|V(G^+)|-1$ if the head of $a$ is in $U$ and the tail of $a$ is not in $U$. Otherwise, Alice outputs $0$.  It is easy to check that there is a bijection between the arcs $a$ for which Alice outputs $|V(G^+)|-1$ and the components of $G^+[U] \cap T$ which do not contain $u$. Therefore, in expectation, the output of the protocol is $k-1$, as desired. The complexity of the protocol is $\log(|V(G^+)|)+\log (|E(G^+)|)+O(1)$. Hence the size of the resulting formulation is in $O(|V(G)| \cdot |E(G)|)$ by Theorem~\ref{thm:random}.

\subsection{Proof of Main Theorem via Protocols} \label{sec:proof}

We now give the second proof of our main theorem.  Unlike the direct proof given in Section~\ref{DirectProof}, this proof is not constructive, since Theorem~\ref{thm:random} does not give an efficient procedure to write down the extended formulation (although such a procedure is known for deterministic protocols \cite{aprile2020extended}). On the other hand, the proof via protocols inspired the direct proof. Indeed, the main idea is that Alice can use less bits to send vertex $u$ if the graph has balanced separators. This led us to investigate how Martin's formulation behaves in the presence of separators, resulting in Lemmas~\ref{lem:2conn} and \ref{lem:vtx_deletion}.

The proof uses the following definitions. For $\beta\in(0,1)$ and $c\in\mathbb{R}^+$, a \defn{$(c, \beta)$-balanced separator tree} for a graph $G$ is a binary tree $T_G$ defined recursively as follows.  The root of $T_G$ is $(G,X)$ where $X$ is a $\frac{1}{2}$-balanced separator of $G$ of order at most $c|V(G)|^\beta$.  Suppose $(H,Y)$ is a node of $T$ such that $H-Y$ is the disjoint union of two graphs $H_1$ and $H_2$ with $|V(H_1)|, |V(H_2)| \leq \frac{1}{2} |V(H)|$.  If $|V(H)| > c$, then $(H,Y)$ has two children $(H_1, Y_1)$ and $(H_2, Y_2)$ where $Y_i$ is a $\frac{1}{2}$-balanced separator of $H_i$ of order at most $c|H_i|^\beta$.  If $|V(H)| \leq c$, then $(H,Y)$ is a leaf of $T$. By Lemma~\ref{sublinear}, if $\GG$ is a graph class closed under induced subgraphs and $\GG$ has strongly sublinear separators, then every graph $G \in \GG$ has a {$(c, \beta)$-balanced separator tree} for some $c \in \mathbb{R}^+$ and $\beta\in(0,1)$.
 
In Edmonds' original description of the forest polytope (Theorem~\ref{thm:LP2}), there is a constraint for \emph{every} proper, non-empty subset $U$ of $V(G)$, but it turns out that we only need those constraints when $G[U]$ is connected (see  \cite[Theorem~40.5]{schrijver2003combinatorial}). The classical protocol does not need that $G[U]$ is connected, but our proof crucially exploits this fact.

\begin{proof}[Second proof of Theorem~\ref{thm:main2}] 
Let $G$ be an $n$-vertex connected graph in $\GG$. We prove the theorem by describing an appropriate randomized protocol. As in the classical protocol, Alice receives a non-empty, proper subset $U$ of $V(G)$, such that $G[U]$ is connected as input, Bob receives a forest $F$ of $G$ as input, and their goal is to compute $S_G(U,F)$ in expectation. The players agree beforehand on a $(c, \beta)$-balanced separator tree $T_G$ of $G$.  In contrast to the classical protocol, the main idea is that Alice does not send a vertex $u\in U$ right away, but uses $T_G$ to delay sending $u$ to Bob, until she can do so using ``few'' bits. 

Let $(G,X)$ be the root of $T_G$.  If $U\cap X\neq \emptyset$, then Alice sends a vertex $u$ of $U\cap X$ to Bob and the protocol proceeds as in the classical protocol. Otherwise, let $(A_1, X_1)$ and $(B_1, Y_1)$ be the children of $(G,X)$.  Since $U\cap X= \emptyset$ and $G[U]$ is connected, either $U\subseteq V(A_1)$ or $U\subseteq V(B_1)$. Alice sends one bit to Bob to signal which case occurs, then she recurses on the corresponding subgraph (say, without loss of generality, $A_1$). 
This creates a path from the root $(G,X):=(A_0, X_0)$ to $(A_t, X_t)$ in $T_G$ where $U\cap X_i=\emptyset$ for $i\in[0,t-1]$, and either $U \cap X_t\neq \emptyset$ or $|V(A_t)| \leq c$. If $U \cap X_t\neq \emptyset$, then Alice sends Bob a vertex $u \in U \cap X_t$.  If $|V(A_t)| \leq c$, then Alice sends Bob an arbitrary vertex $u \in U$.  

Let $F_t:=F \cap A_t$. 
Observe that $S_{G}(U, F)=S_{A_t}(U, F_t)$. 
Hence, at this point Alice and Bob can proceed as in the classical protocol, but with the graph $A_t$ and the forest $F_t$ instead, to compute (in expectation) the slack $S_{A_t}(U, F_t)$. 

It remains to analyze the complexity of the protocol.  Alice spends $t-1$ bits to tell Bob the path in $T_G$ from $(G,X)$ to $(A_t, X_t)$.  Since $u \in X_t$ or $|V(A_t)| \leq c$, Alice can send $u$ to Bob using at most $\log c |V(A_t)|^\beta \leq \log c n^\beta$ bits. By assumption, $|E(A_t)| \leq d |V(A_t)|$. Therefore, $|E(A_t^+)| \leq (d+1)|V(A_t)|$, where $A_t^+$ denotes the graph $A_t$ plus a universal vertex.  Since $|V(A_t)| \leq \frac{1}{2^t}|V(G)|$, we have $|E(A_t^+)| \leq \frac{d+1}{2^t}n$.  Finally, since the arc $a$ can be oriented in two ways, there are at most $\frac{d+1}{2^{t-1}}n$ possibilities for the choice of the arc $a$ chosen by Bob after switching to the classical protocol.  Thus, Bob can send $a$ to Alice using at most $\log (\frac{d+1}{2^{t-1}}n)$ bits. Therefore, the total amount of communication required for the protocol is at most
\begin{align*}
(t-1) \,+\, \lceil \log\left(cn^\beta\right) \rceil \,+\, \lceil \log\left( \tfrac{d+1}{2^{t-1}}n\right) \rceil  \,= \, 
\log\left(c(d+1)n^{1+\beta}\right) + O(1).
\end{align*}
Since $S_G$ is a submatrix of the slack matrix of $\FP(G)$ with respect to the linear description given in Theorem~\ref{thm:LP2}, and there are only $|E(G)|=O(dn)$ additional rows of the slack matrix, it follows from Theorem~\ref{thm:random} that $\xc(\FP(G)) \in O(cdn^{1+\beta})$, as required.  \end{proof}

\subsection{Planar graphs} \label{sec:planar}
 
As mentioned in the introduction,  \citet{Williams2002} proved that for every connected planar graph $G$ on $n$ vertices, $\STP(G)$ has an extended formulation of size $O(n)$. We reprove this result by giving a randomized protocol for (the non-trivial part of) the slack matrix of $\STP(G)$.  The main idea is exploiting the relationship between the spanning trees of a planar graph and its dual. 

Let $G$ be a planar graph, with a fixed embedding. Let $F(G)$ be the set of faces of $G$. The \defn{dual} $G^*$ is the multigraph with $V(G^*):=F(G)$, where for each edge $e=vw$ of $G$ incident with faces $f$ and $g$, there is an  edge $e^*=fg$ in $G^*$, called the \defn{dual} of $e$. For a set $X\subseteq E(G)$, let $X^*$ be the set of edges of $G^*$ dual to the edges in $X$. It is well known that $X\subseteq E(G)$ is the edge-set of a spanning tree of $G$ if and only if $E(G^*)\setminus X^*$ is the edge-set of a spanning tree of $G^*$. For any spanning tree $T$ of $G$, let $T^*$ be the dual spanning tree of $G^*$ with $E(T^*):= E(G^*)\setminus E(T)^*$. Note that the definitions of $\STP(G)$ and $\FP(G)$ extend to the setting of multigraphs. 
In particular, letting $\varphi: \R^E\rightarrow \R^E$ with $\varphi(x)=\mathbf{1}-x$, we have that $\STP(G^*)=\varphi(\STP(G))$.
This implies a one-to-one correspondence between the vertices of the two polytopes and between their facets. The next lemma shows that this isomorphism preserves the slack. 

Recall that, by Lemma~\ref{lem:2conn}, we may assume that $G$ is 2-connected. Indeed, if $G$ is not 2-connected and $G_1, \dots, G_k$ are the blocks of $G$, then since $|V(G_1)|+\dots+|V(G_k)|=|V(G)|+k-1\leq 2|V(G)|$, a bound of the form $\xc(\STP(G_i))\leq c|V(G_i)|$ for each $i\in[k]$ and for some constant $c$ implies $\xc(\STP(G))\leq 2c|V(G)|$. It is known (see \cite{feichtner2005matroid}) that for any 2-connected graph $G$, the facets of $\STP(G)$ are defined by the non-empty sets $U\subsetneq V(G)$ such that both $G[U]$ and $G/U$ are 2-connected. Here $G/U$ is the graph obtained from $G$ by contracting the subgraph $G[U]$ to a single vertex. 

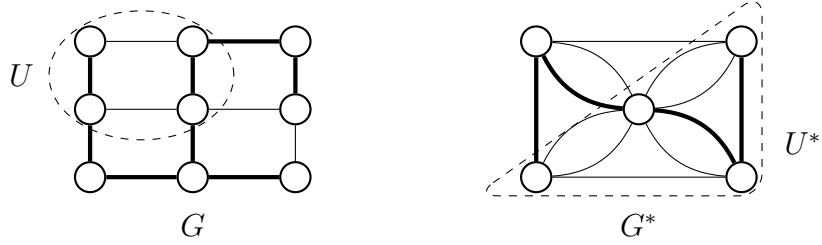
\begin{figure}[htbp]
\centering
\begin{tikzpicture}[inner sep=4pt,scale=.9]
\tikzstyle{vtx}=[circle,draw,thick];
\tikzstyle{vtxg}=[circle,draw,thick,fill=black!50];
\tikzstyle{vtxr}=[circle,draw,thick,fill=red!50];

\draw[dashed] (-0.75, 0.5) ellipse (1.35cm and 0.95cm);
\node at (-2.5,0.5) {$U$};

\node[vtx] (v1) at (-1.5,-1) {};
\node[vtx] (v2) at (-1.5,0) {};
\node[vtx] (v3) at (-1.5,1) {};

\node[vtx] (v4) at (0,-1) {};
\node[vtx] (v5) at (0,0) {};
\node[vtx] (v6) at (0,1) {};

\node[vtx] (v7) at (1.5,-1) {};
\node[vtx] (v8) at (1.5,0) {};
\node[vtx] (v9) at (1.5,1) {};

\draw[ultra thick] (v1) -- (v2);
\draw[ultra thick] (v2) -- (v3);
\draw[ultra thick] (v4) -- (v5);
\draw[ultra thick] (v5) -- (v6);
\draw (v7) -- (v8);
\draw[ultra thick] (v8) -- (v9);
\draw[ultra thick] (v1) -- (v4);
\draw (v2) -- (v5);
\draw (v3) -- (v6);
\draw[ultra thick] (v4) -- (v7);
\draw (v5) -- (v8);
\draw[ultra thick] (v6) -- (v9);
\node at (0,-1.7) {$G$};

\end{tikzpicture}\hspace{2cm}
\begin{tikzpicture}[inner sep=4pt,scale=.9]
\tikzstyle{vtx}=[circle,draw,thick];
\tikzstyle{vtxg}=[circle,draw,thick,fill=black!50];
\tikzstyle{vtxr}=[circle,draw,thick,fill=red!50];

\draw [dashed, rounded corners=3mm, scale=1.5, xshift = -0.2cm, yshift = +0.15cm] (-1.4,-1)--(1.4,1)--(1.4,-1)--cycle;
\node at (2.4,-0.5) {$U^*$};
\node at (0,-1.7) {$G^*$};
\node[vtx] (v1) at (-1.5,-1) {};

\node[vtx] (v3) at (-1.5,1) {};

\node[vtx] (v5) at (0,0) {};

\node[vtx] (v7) at (1.5,-1) {};

\node[vtx] (v9) at (1.5,1) {};

\draw[ultra thick](v1) -- (v3);
\draw (v1) -- (v7);
\draw[ultra thick] (v7) -- (v9);
\draw (v3) -- (v9);
\draw (v1) to [bend right] (v5);
\draw (v1) to [bend left]  (v5);
\draw[ultra thick] (v3) to [bend right] (v5);
\draw (v3) to [bend left]  (v5);
\draw[ultra thick] (v7) to [bend right] (v5);
\draw (v7) to [bend left]  (v5);
\draw (v9) to [bend right] (v5);
\draw (v9) to [bend left]  (v5);

\end{tikzpicture}
\caption{On the left, a planar graph $G$, a set $U$ defining a facet of $\STP(G)$ and a spanning tree of $G$ (thick edges). On the right, the dual graph $G^*$ (with the vertex corresponding to the outer face of $G$ drawn in the center), together with $U^*$ and $T^*$ (thick edges). The equation in Lemma \ref{lem:planar} can be easily verified for this example. } 
\label{fig:planar}
\end{figure}

\begin{lemma}\label{lem:planar}
Let $G$ be a 2-connected planar graph with a fixed embedding. Consider a facet of $\STP(G)$ defined by a non-empty subset $U\subsetneq V(G)$. Let $U^*$ be the set of vertices of $G^*$ corresponding to faces of $G$ that have at least one vertex not in $U$. Then, for any spanning tree $T$ of $G$, 
$$
|U|-1-|E(T)\cap E(U)|=|U^*|-1-|E(T^*)\cap E(U^*)|.
$$
Hence, the facet of $\STP(G^*)$ defined by $U^*$ corresponds (under $\varphi$) to the facet of $\STP(G)$ defined by $U$, and the two facets have the same slack.
\end{lemma}

\begin{proof}
 We first prove that an edge $e=uv\in E(G)$ is in $E(U)$ if and only if its dual edge $e^*\in E(G^*)$ is not in $E(U^*)$. Let $f,g\in F(G)$ be the faces incident to $e$. First, if $e^*\not\in E(U^*)$, then all the vertices of $f$ and $g$, in particular $u$ and $v$, are in $U$, implying that $e\in E(U)$. Now assume that $e\in E(U)$; that is, $u,v\in U$. We show that $e\not\in E(U^*)$, equivalently that one of $f$ and $g$ is not in $U^*$. Aiming for a contradiction, assume that there are vertices $w,w'\not\in U$ that are on the boundaries of $f$ and $g$, respectively. Recall that $G[U]$ and $G/U$ are 2-connected. First, since $G[U]$ is 2-connected, there is a $uv$-path $P$ in $G[U]$ that does not contain $e$. (Note that here we only need that $G[U]$ is 2-edge-connected.)\ Let $C$ be the cycle made by $P$ and $e$. Without loss of generality, $f$ is contained in the interior of $C$ and $g$ is contained in the exterior of $C$. Since $w$ and $w'$ are not in $U$, $w$ is strictly in the interior of $C$ and $w'$ is strictly in the exterior of $C$. Thus $C$ separates $w$ and $w'$. Since every edge of $C$ is in $E(U)$, this contradicts the assumption that $G/U$ is 2-connected. Thus $e$ is in $E(U)$ if and only if $e^*$ is not in $E(U^*)$. This shows that $U^*$ defines the facet of $\STP(G^*)$ corresponding (via the isomorphism $\varphi$) to the facet of $\STP(G)$ defined by $U$. Moreover,  $|E(G)|=|E(U)|+|E(U^*)|$. 
  
We now determine the number of faces in the embedding of $G[U]$ induced by that of $G$. Consider $G$ and $G[U]$ to be embedded in the sphere. We claim that there is exactly one face of $G[U]$ that is not a face of $G$. Since $U\neq V(G)$, there is at least one such face. Suppose that $f$ and $g$ are distinct faces of $G[U]$ that are not faces of $G$. Let $D_f$ and $D_g$ be the discs associated with $f$ and $g$. Then there is a vertex $v\in V(G)\setminus U$ in the interior of $D_f$, and there is a vertex $w\in V(G)\setminus U$ in the interior of $D_g$. Thus $v$ and $w$ are separated by $U$, which contradicts the assumption that $G/U$ is 2-connected. Thus there is exactly one face of $G[U]$ that is not a face of $G$. 
 Each face of $G$ either has all its vertices in $U$, and is thus a  face of $G[U]$, or it  corresponds to a vertex in $U^*$. With the above claim, this shows that $G[U]$ has $|F(G)|-|U^*|+1$ faces. 
  
  We now prove the thesis. 
  By Euler's formula applied to $G[U]$ and to $G$,
\begin{align*}
|U| - |E(U)| + ( |F(G)|-|U^*|+1 ) = 2 = |V(G)| - |E(G)| + |F(G)|.
\end{align*}
Since  $|E(G)|=|E(U)|+|E(U^*)|$,
\begin{align*}
|U|  = |U^*| + |V(G)|-1 - |E(U^*)|.
\end{align*}
For any spanning tree $T$ of $G$, since $|V(G)|-1 = |E(T)| = |E(T)\cap E(U)| + |E(T)\setminus E(U)|$,
\begin{align*}
|U|-|E(T)\cap E(U)| = |U^*|     + |E(T)\setminus E(U)| - |E(U^*)|.
\end{align*}
Observe that the edges of $E(T)\setminus E(U)$ are dual to the edges of $E(U^*)\setminus E(T^*)$. Thus
\begin{align*}
|U|-1-|E(T)\cap E(U)| 
& = |U^*|-1     + |E(U^*)\setminus E(T^*)| - |E(U^*)|\\
& = |U^*|-1     -|E(U^*)\cap E(T^*)|.
\end{align*}
Hence the facet of $\STP(G^*)$ defined by $U^*$ has the same slack as the corresponding facet of $\STP(G)$ defined by $U$. 
\end{proof}

 We now describe a protocol for $\STP(G)$ when $G$ is an $n$-vertex  2-connected planar graph, based on a simple modification of the classical protocol described in Section~\ref{sec:protocols}. Alice receives as input a non-empty set $U \subsetneq V(G)$, such that $G[U]$ and $G/U$ are 2-connected, and Bob receives a spanning tree $T$ of $G$ (since we consider the spanning tree polytope of $G$ here). Fix a vertex $v_0$ and a face $f_0$ of $G$, such that the boundary of $f_0$ contains $v_0$. The crucial observation is that, for any $U$, if $v_0\not\in U$, then $f_0\in U^*$ by definition. Hence Alice, instead of sending a vertex $u\in U$, just sends one bit indicating whether $v_0\in U$ or $f_0\in U^*$. In the first case, the protocol proceeds as in the classical protocol. In the second case, Alice and Bob switch to the dual graph $G^*$, where $U$ is replaced by $U^*$ and $T$ by its dual $T^*$, and go on with the classical protocol. Note that $G^*$ is not necessarily a simple graph, but the classical protocol does not require this. The correctness of the protocol in the second case is guaranteed by Lemma~\ref{lem:planar}. The complexity of the protocol is at most \[
 1+ \max\{\lceil \log ( 2|E(G)|) \rceil, \lceil \log (2|E(G^*)|) \rceil\}=1+\lceil \log ( 2|E(G)|) \rceil \leq \log(n)+O(1).
 \] 
By Theorem~\ref{thm:random}, the size of the resulting extended formulation is $O(n)$, thus matching Williams' result.

 \section{Open Problems}
 \label{sec:open_problems}
 
 We have shown that for every proper minor-closed graph class $\GG$, the spanning tree polytope of every connected $n$-vertex graph in $\GG$ has extension complexity in $O(n^{3/2})$. This gives some evidence in support of the following conjecture of \citet{FHJP17}. 

\begin{conjecture}[\cite{FHJP17}]
For every proper minor-closed graph class $\GG$, the spanning tree polytope of every connected $n$-vertex graph in $\GG$ has extension complexity in $O(n)$.
\end{conjecture}

Much stronger results might hold. Is it true that $\xc(\STP(G))\in O(|V(G)|)$ for connected graphs $G$ in:
\begin{itemize}[itemsep=0ex,topsep=1ex]
\item graph classes admitting strongly sublinear separators, 
\item graph classes with bounded expansion,
\item graph classes with bounded maximum degree, or
\item graph classes with bounded density?
\end{itemize}

Given the lack of lower bounds, a first question is whether $\xc(\STP(G))\in O(|V(G)|)$ for a random cubic graph $G$. Another example of interest is the 1-subdivision of a complete graph.

Another possible direction of investigation is the extension complexity of \defn{matroid base polytopes}, of which spanning tree polytopes are a special case (corresponding to graphic matroids).  \citet{rothvoss2013some} proved (via a counting argument) that there are matroid base polytopes with exponential extension complexity.  However, there is no known \emph{explicit} family of matroid base polytopes with super-polynomial extension complexity. On the other hand, polynomial-size extended formulations are known for some classes that strictly contain graphic matroids; see \citep{aprile2019regular, aprile2021extended, aprile20182, conforti2015subgraph}. All such formulations have deep roots in Martin's and Wong's formulations for the spanning tree polytope, which is a further reason to investigate its extension complexity.

\section*{Acknowledgments}

Tony Huynh, Gwena\"{e}l Joret and David Wood are supported by the Australian Research Council.  Gwena\"{e}l Joret is also supported by an ARC grant from the Wallonia-Brussels Federation of Belgium, and a CDR grant from the Belgian National Fund for Scientific Research (FNRS).  Samuel Fiorini is supported by the FNRS, through PDR grant BD-OCP/T.0087.20. This work was partially supported by ERC Consolidator grant FOREFRONT/615640. Manuel Aprile is supported by a SID 2019 grant of the University of Padova.

  \let\oldthebibliography=\thebibliography
  \let\endoldthebibliography=\endthebibliography
  \renewenvironment{thebibliography}[1]{%
    \begin{oldthebibliography}{#1}%
      \setlength{\parskip}{0ex}%
      \setlength{\itemsep}{0ex}%
  }{\end{oldthebibliography}}

\end{document}